\documentclass[11pt]{article}
\usepackage{amsmath}
\usepackage{amssymb}
\usepackage{amsthm}
\usepackage{lineno}
\usepackage[mathscr]{eucal}
\usepackage{graphicx}
\theoremstyle{definition}
\newtheorem{definition}{Definition}
\newtheorem{theorem}[definition]{Theorem}
\newtheorem{proposition}[definition]{Proposition}

\newtheorem{corollary}[definition]{Corollary}
\theoremstyle{remark}
\newtheorem{remark}[definition]{Remark}
\newtheorem{example}[definition]{Example}
\newcounter{enumctr}

\newcommand{\N}{\mathbb{N}}
\newcommand{\R}{\mathbb{R}}
\newcommand{\Q}{\mathbb{Q}}

\newcommand{\Z}{\mathbb{Z}}

\renewcommand{\phi}{\varphi}

\begin{document}
%
%
%
\title{\vspace*{-10mm}
Semigroup property of fractional differential operators and its applications}
\author{
N.D.~Cong\footnote{\tt ndcong@math.ac.vn, \rm Institute of Mathematics, Vietnam Academy of Science and Technology, 18 Hoang Quoc Viet, 10307 Ha Noi, Viet Nam.}
}

\date{}

\maketitle
\begin{abstract}
We establish partial semigroup property of Riemann-Liouville and Caputo fractional differential operators. Using this result we prove theorems on reduction of multi-term  fractional differential systems to single-term and multi-order systems, and prove existence and uniqueness of solution to multi-term Caputo fractional differential systems.
\end{abstract}

\section{Introduction}

The theory of fractional differential equations has been developed fast in the last few decades due to the contributions of scientists from various fields of research (see \cite{Kai,KilbasSriTru06,Podlubny} and the reference therein). With its long memory and nonlocal feature the fractional differential equations serve better to modeling real world precesses in many applications from physics, engineering to biology and finance. It is of great importance, both from the theoretical and practical point of view, to understand the fundamental properties of fractional differential operators, especially in comparison with the well investigated properties of the classical integer order differential operators. 

In this paper we deal with one of the fundamental properties of the fractional differential operators, namely we study the semigroup property of Riemann-Liouville and Caputo fractional differential operators.
It is known that one of the important property of the classical differentiation is its semigroup property: if a function $f\in C^m[a,b]$, $[a,b]\subset\R$, $m\in\N$, i.e. $f$ is $m$ times continuously differentiable on $[a,b]$, then for all $n,l\in\N$, $n+l\leq m$ we have $f\in C^n[a,b]$, $D^n f\in C^l[a,b]$ and $D^l(D^n f)=D^{n+l} f$. This semigroup property allows us to reduce a high order ordinary differential equation  to a high-dimensional system of ordinary differential equations of order one, for which  a well established theory is already developed. As for the fractional differential operators it is well known that the semigroup property does, in general, not hold (see \cite[Examples 2.6 - 2.7, Remark 3.4]{Kai}, \cite[Sections 2.3.6, 2.4.1]{Podlubny}). Diethelm~\cite[Lemma 3.13]{Kai} has proved a kind of weak semigroup property of the Caputo fractional differential operators and used it to derive a theorem on 
reduction of a multi-term Caputo fractional differential equation to a single-term Caputo fractional differential equation \cite[Theorem 8.1]{Kai}. This result is a useful tool for studying multi-term fractional differential equations, since it allows us to reduce a multi-term system to a single-term system for which many tools of investigation are available. There are a number of works dealing with this problem, especially the reduction to single-term fractional differential equations is used to study the asymptotic behavior of solutions to multi-order fractional differential systems \cite{Kai, DieSieTuan2017, Badri2019, Deng2007, Li2013, Rebenda2019}. However, for a proof of  this weak semigroup property of Caputo fractional differential operators by Diethelm~\cite[Lemma 3.13, Theorems 8.1, 8.2, 8.9]{Kai}, the author requires a strong assumption of smoothness of solution. This high smoothness assumption seems unnatural, and can hardly be fulfilled for a general Caputo fractional differential equation. For a simple example let us consider an initial value problem with Caputo fractional differential equation
  $$
{}^C\!D_{0}^{\alpha} x(t) - \lambda x(t) =0, \quad t\in [0,1], 0<\alpha<1, \lambda\in\R, x(0)=x_0\not=0.
$$
The solution to this problem is the Mittag-Leffler function $x_0E_\alpha(\lambda t^\alpha)$ (see Diethelm~\cite[Theorem 7.2]{Kai}), which is clearly continuous on $[0,1]$ but its derivative,  $\frac{d}{dt}x_0E_\alpha(\lambda t^\alpha) = x_0\lambda t^{\alpha-1}E_{\alpha,\alpha}(\lambda t^\alpha)$, has a singularity at 0, hence the solution $x_0E_\alpha(\lambda t^\alpha)\not\in C^1[0,1]$, thus does not have the degree of smoothness required by Diethelm~\cite[Lemma 3.13]{Kai} for this Caputo fractional differential equation. 
The problem is more visible if we replace the constant $\lambda$ in the above equation by a piecewise constant or piecewise continuous function $\lambda(t)$, what can easily lead to non-differentiability of the solution in points in the interior of the interval $[0,1]$.
This shows that the strong assumption of  smoothness is hardly satisfied, so that Lemma 3.13 of Diethelm~\cite{Kai}, hence Theorems 8.1, 8.2, 8.9 therein, have limited applicability to the Caputo fractional differential equations. This example also shows that a natural requirement for the solutions of initial value problem to the Caputo fractional differential equations is continuous Caputo fractional differentiability of order $\alpha$---the order of the fractional differential equation, and that is what we do in this paper. 

In this paper we will prove a (partial) semigroup property for the Riemann-Liouville as well as for the Caputo fractional differential operators under natural assumptions of fractional differentiability of the function under the operators. 

 The paper is organized as follow. In Section~\ref{sec.preliminaries} we present the definitions of Riemann-Liouville and Caputo fractional differential operators which will be used in this paper. We present two approaches to the definition of the operators: the integrable and the continuous setting. In each case we describe the domain and the range of the operator in details. In Section~\ref{sec.RelDifOrders} we present the main results of the paper: the (partial) semigroup property of Riemann-Liouville and Caputo fractional differential operators. The results in this section provide us with a complete description of a (partial) semigroup property of the Riemann-Liouville/Caputo fractional differential operators. In Section~\ref{sec.applications} we present two important applications of the results on semigroup property derived in Section~\ref{sec.RelDifOrders}, namely the reduction of multi-term Caputo fractional differential equations to single-term and multi-order Caputo fractional differential systems, and a proof of the existence and uniqueness of solution to multi-term Caputo fractional differential equation. Section~\ref{sec.examples} presents two illustrative examples for the results of the preceeding sections.

   \section{Preliminaries}\label{sec.preliminaries}
   Recall the notion of Riemann-Liouville fractional integral operator from Diethelm~\cite[Definition 2.1, p. 13]{Kai}.     Let   $[a,b]\subset \R$, $-\infty<a<b<\infty$, be a closed interval of the real line.  
For $\alpha>0$,   the operator  $J_a^\alpha$ defined on $L_1[a,b]$ by
   $$
   J_a^\alpha f(x) := \frac{1}{\Gamma(\alpha)}\int_a^x (x-t)^{\alpha-1} f(t) dt,
   $$
   for $a\leq x\leq b$ is called {\em Riemann-Liouville fractional integral operator of order $\alpha$}. For $\alpha=0$, we set $J_a^0=I$, the identity operator.
   
   The integral operator $J_a^\alpha$ has semigroup property with respect to its parameter $\alpha\geq 0$ 
   \cite[Theorem 2.2, p. 14]{Kai},
   $$
    J_a^\alpha  J_a^\beta =  J_a^{\alpha+\beta}\quad \hbox{for all}\quad \alpha>0,\beta >0.
    $$
    Furthermore, clearly  $J_a^\alpha f(x)\in C[a,b]$, hence the range of the operator $J_a^\alpha$, $\alpha>0$, is in $C[a,b]$.
     Clearly, $C[a,b]\subset L_1[a,b]$ hence $J_a^\alpha$ can also be considered as an operator from $C[a,b]$ to $C[a,b]$ for all $\alpha\geq 0$.

In this paper we deal with two well-known fractional differential operators, Riemann-Liouville and Caputo fractional differential operators, which are defined based on the Riemann-Liouville integral operator.  Like the integral operator we can consider the differential operators, which are inverses of  the integral operator, in two settings: integrable and continuous. 
 
Now we give a definition of Riemann-Liuoville fractional differential differentiation in both the integrable and continuous settings  (cf. Diethelm~\cite[Definition 2.2, p. 27]{Kai} and Vainikko~\cite[Section 4]{Vainikko2016}).

Denote by $\R,\Q,\Z,\N$ the sets of real numbers, rational numbers, integers and natural numbers, respectively.
  For a real number  $\alpha\in\R$ we denote by $\lceil \alpha\rceil$ the ceilling of $\alpha$, i.e. smallest integer no smaller $\alpha$, by $\lfloor\alpha\rfloor$ the floor of $\alpha$, i.e. the greatest integer no greater than $\alpha$. 

   \begin{definition}[Riemann-Liouville fractional derivative]\label{def.RL}
   Let $0<\alpha\in \R$ and $m:= \lceil \alpha\rceil$. The operator ${}^{RL}\!D_{a}^{\alpha}$ defined on 
   $L_1[a,b]$ by
   $$
   {}^{RL}\!D_{a}^{\alpha}f := D^m J_a^{m-\alpha} f,
   $$
 where $D^m$ is the ordinary differential operator of integer order $m$, is called the {\em Riemann-Liouville fractional operator of order $\alpha$}. For $\alpha=0$, we set ${}^{RL}\!D_{a}^{0}=I$, the identity operator.
 
 If the function $D^m J_a^{m-\alpha} f$ exists and is in the class 
 $L_1[a,b]$  then we say that the function $f$ is {\em Riemann-Liouville $\alpha$-differentiable}. 
  
   If  $f\in C[a,b]$ and
   $D^m J_a^{m-\alpha} f$ exists and is in the class $C[a,b]$ then we say that $f$ is {\em continuously Riemann-Liouville $\alpha$-differentiable}.
   \end{definition}
   
   So, in the integrable setting we consider ${}^{RL}\!D_{a}^{\alpha}: L_1[a,b]\to L_1[a,b]$, whereas in the continuous setting ${}^{RL}\!D_{a}^{\alpha}: C[a,b]\to C[a,b]$.
     Cleary, for each fixed $\alpha>0$,  there are functions $f\in L_1[a,b]$ which are not Riemann-Liouville $\alpha$-differentiable, and there are functions $f\in C[a,b]$ which are neither continuously Riemann-Liouville $\alpha$-differentiable nor Riemann-Liouville $\alpha$-differentiable.

   Now we recall a definition of Caputo fractional differential differentiation in both the integrable and continuous settings  (cf. Diethelm~\cite[Definition 3.2, p. 50]{Kai} and 
   Vainikko~\cite[Section 5]{Vainikko2016}). 
   \begin{definition}[Caputo fractional derivative]\label{def.Cap}
   Let $0 < \alpha\in \R$ and $m:= \lceil \alpha\rceil$. Moreover, assume that $f\in C^{m-1}[a,b]$ is such that 
 ${}^{RL}\!D_{a}^{\alpha}f$ exists and belongs to $L_1[a,b]$.   The operator ${}^{C}\!D_{a}^{\alpha}$ defined on 
   $C^{m-1}[a,b]$ by
   $$
   {}^{C}\!D_{a}^{\alpha}f := {}^{RL}\!D_{a}^{\alpha}(f- T_{m-1}[f;a]),
   $$
   where $T_{m-1}[f;a]:= \sum_{k=0}^{m-1} \frac{(D^kf)(a)}{k!}(t-a)^k$ denotes the Taylor polynomial of order $m-$ of $f$ centered at $a$,
  is called the {\em Caputo fractional operator of order $\alpha$}.  For $\alpha=0$, we set ${}^{C}\!D_{a}^{0}=I$, the identity operator.
 
 If the function ${}^{C}\!D_{a}^{\alpha}f$ exists and is in the class 
 $L_1[a,b]$  then we say that the function $f$ is {\em Caputo $\alpha$-differentiable}. 
  
   If  ${}^{C}\!D_{a}^{\alpha}f$ exists and is in the class $C[a,b]$ then we say that $f$ is {\em continuously Caputo $\alpha$-differentiable}.
   \end{definition}
   
   In parallel with the Riemann-Liouville fractional differentiation we see that in the integrable functions setting 
   ${}^{C}\!D_{a}^{\alpha}: C^{m-1}[a,b]\to L_1[a,b]$, whereas in the continuous functions setting 
   ${}^{C}\!D_{a}^{\alpha}: C^{m-1}[a,b]\to C[a,b]$.
 Note that in case $f\in C^m[a,b]$ there exists continuous ${}^{C}\!D_{a}^{\alpha}f $, and if additionally $\alpha\not\in\N$ then ${}^{C}\!D_{a}^{\alpha}f (a)=0$, see Diethelm~\cite[Lemma 3.11, p. 56]{Kai}.

   Notice that in Definition~\ref{def.Cap} we require the function $f\in C[a,b]$ to be of class $C^{\lceil \alpha\rceil-1}[a,b]$, hence ${}^{C}\!D_{a}^{\alpha} : C^{\lceil \alpha\rceil-1}[a,b]\to C[a,b]$; whereas some other authors (including Caputo in his original work \cite{Caputo67}) require the function $f$ to be of class $C^{\lceil \alpha\rceil}[a,b]$ to define its Caputo $\alpha$-derivative by a formula other than that of Definition~\ref{def.Cap}, hence in their approach ${}^{C}\!D_{a}^{\alpha} : C^{\lceil \alpha\rceil}[a,b]\to C[a,b]$. The first approach is a generalization of the second one, and in case $f\in C^{\lceil \alpha\rceil}[a,b]$ the two definitions coincide.
   
   Unlike the fractional integral operators $J_a^\alpha$, $0\leq \alpha\in\R$, and differential operators of integer orders $D^m$, $m\in\N$, the fractional differential operators, both Riemann-Liouville ${}^{RL}\!D_{a}^{\alpha}f$ and Caputo ${}^{C}\!D_{a}^{\alpha}f $, $0\leq\alpha\in\R$, do not have semigroup property with respect to the parameter $\alpha$ of differentiation. They may possess semigroup property under some additional assumptions, see Diethelm~\cite[Theorem 2.13, p. 29]{Kai} for the Riemann-Liouville case,  Diethelm~\cite[Lemma 3.13, p. 56]{Kai}  for the Caputo case; and Section \ref{sec.RelDifOrders} in this paper.
   
   \section{Semigroup property of Riemann-Liouville and Caputo fractional differential operators}\label{sec.RelDifOrders}

In this section we investigate in details semigroup property for Riemann-Liouville ${}^{RL}\!D_{a}^{\alpha}f$ and Caputo fractional differential operators ${}^{C}\!D_{a}^{\alpha}f$. The semigroup property derived is partial because it is not for arbitrary order of differentiation $0\leq \alpha\in\R$ but only true for a particular subset of the parematers  (see below for details). However, the theorems we get are true for all the functions of suitable $\alpha$-differentiability, similar to the case of differentiation of integer order.
    
      \subsection{Semigroup property of Riemann-Liouville fractional differential operators}\label{subsec.RL}
   
     \begin{theorem}[Partial semigroup property of Riemann-Liouville fractional differential operators]\label{thm.RL1}
    I. Let $0\leq\beta<\alpha < 1$ be arbitrary, $f\in C[a,b]$ be continuously Riemann-Liouville $\alpha$-differentiable. Then 
    
   (i) $f$ is continuously Riemann-Liouville $\beta$-differentiable with  ${}^{RL}\!D_{a}^{\beta}f\in C[a,b]$. Moreover,  $({}^{RL}\!D_{a}^{\beta}f)(a) = 0$.
   
 (ii)  The continuous fractional derivative $f_1:={}^{RL}\!D_{a}^{\beta}f\in C[a,b]$ is continuously Riemann-Liouville $(\alpha-\beta)$-differentiable. Moreover,
  ${}^{RL}\!D_{a}^{\alpha-\beta}f_1= {}^{RL}\!D_{a}^{\alpha}f$, thus 
 ${}^{RL}\!D_{a}^{\alpha-\beta}({}^{RL}\!D_{a}^{\beta}f)= {}^{RL}\!D_{a}^{\alpha}f$. 
 
 (iii) If $f\in C^1[a,b]$ and $f(a)=0$ then for any $0\leq \beta\leq 1$ the function $f$ is continuously Riemann-Liouville $\beta$-differentiable with derivative $f_1={}^{RL}\!D_{a}^{\beta}f\in C[a,b]$. Moreover, 
  ${}^{RL}\!D_{a}^{1-\beta}f_1= Df$, thus 
 ${}^{RL}\!D_{a}^{1-\beta}({}^{RL}\!D_{a}^{\beta}f)= Df$.
 
 II. Conversely, let $0<\beta<1$, $\gamma>0$ be arbitrary. Assume that a continuous function $h\in C[a,b]$ is continuously Riemann-Liouville $\beta$-differentiable with derivative ${}^{RL}\!D_{a}^\beta h =: h_1$, and $h_1$ is continuously Riemann-Liouville $\gamma$-differentiable, then $h$ is continuously Riemann-Liouville $(\beta+\gamma)$-differentiable, and 
 ${}^{RL}\!D_{a}^{\gamma}h_1= {}^{RL}\!D_{a}^{\beta+\gamma}h$, thus 
 ${}^{RL}\!D_{a}^{\gamma}({}^{RL}\!D_{a}^{\beta}h)= {}^{RL}\!D_{a}^{\beta+\gamma}h$.
   \end{theorem}
   \begin{proof}
I. (i) Since $f$ is continuously Riemann-Liouville $\alpha$-differentiable we have $J_{a}^{1-\alpha}f \in C^1[a,b]$. Put 
$g:= J_{a}^{1-\alpha}f$. Then $g\in C^1[a,b]$, and $g(a)=0$ because $1-\alpha>0$. Since $\alpha-\beta>0$, by Miller and Ross~\cite[Theorem 2(b), p. 59]{MillerRoss93} we have
$$
D(J_a^{\alpha-\beta}g) = J_a^{\alpha-\beta}(Dg) +\frac{(t-a)^{\alpha-\beta-1} g(a)}{\Gamma(\alpha-\beta)}
= J_a^{\alpha-\beta}(Dg).
$$
Consequently, $J_a^{\alpha-\beta}g\in C^1[a,b]$. But, due to group property of the integral operator $J_a$ we have
$$
J_a^{\alpha-\beta}g = J_a^{\alpha-\beta}J_{a}^{1-\alpha}f = J_a^{1-\beta}f,
$$
Hence $J_a^{1-\beta}f$ is of $C^1[a,b]$, thus $f$ is Riemann-Liouville $\beta$-differentiable with continuous derivative ${}^{RL}\!D_{a}^{\beta}f\in C[a,b]$. 

Now, since  
$$
{}^{RL}\!D_{a}^{\beta}f = D(J_a^{1-\beta}f) = D(J_a^{\alpha-\beta}g) = J_a^{\alpha-\beta}(Dg)
$$
we have $({}^{RL}\!D_{a}^{\beta}f)(a) = J_a^{\alpha-\beta}(Dg)(a)=0$ because $Dg\in C[a,b]$ and $\alpha-\beta>0$.

(ii)  Since 
$f_1= {}^{RL}\!D_{a}^{\beta}f = J_a^{\alpha-\beta}(Dg)$, $g\in C^1[a,b]$ and  $g(a)=0$, due to the group property of the fractional Riemann-Liouville integral, we have 
$$
J_a^{1-(\alpha-\beta)}f_1 = J_a^{1-(\alpha-\beta)}J_a^{\alpha-\beta}(Dg) = J_a^{1}(Dg)=g \in C^1[a,b].
$$
Consequently, $J_a^{1-(\alpha-\beta)}f_1\in C^1[a,b]$, hence
 $f_1$ is continuously Riemann-Liouville $(\alpha-\beta)$-differentiable. 
Now, we have  
$$
{}^{RL}\!D_{a}^{(\alpha-\beta)}f_1 = D(J_a^{1-(\alpha-\beta)}f_1) = Dg = D(J_{a}^{1-\alpha}f)
= {}^{RL}\!D_{a}^{\alpha}f.
$$ 

(iii) If $\beta=1$ or $\beta=0$ then (iii) is obviously true. Suppose that $0<\beta<1$.  From the assumption $f\in C^1[a,b]$ and $f(a)=0$, and due to Diethelm~\cite[Lemma 2.12, p. 27]{Kai}, it implies that for any $0 <\beta< 1$ the function $f$ is continuously Riemann-Liouville $\beta$-differentiable. Use the arguments of (i) and (ii) above one can easily get the assertions of (iii).

II. By assumption $J_a^{1-\beta} h\in C^1[a,b]$ and $h_1= D(J_a^{1-\beta} h)\in C[a,b]$, hence obviously $(J_a^{1-\beta} h)(a)=0$. Put $k:=\lceil\gamma\rceil\geq 1$ then $k\leq \lceil\beta+\gamma\rceil\leq k+1$. From the assumption that $h_1$ is continuously Riemann-Liouville $\gamma$-differentiable it follows that
\begin{equation}\label{eqn.RL1}
J_a^{k-\gamma}h_1\in C^k[a,b] \quad\hbox{and}\quad {}^{RL}\!D_{a}^{\gamma}h_1 = D^k J_a^{k-\gamma}h_1.
\end{equation}
Taking into account the fact that $J_a^{1-\beta} h\in C^1[a,b]$ and $(J_a^{1-\beta} h)(a)=0$, by  Miller and Ross~\cite[Theorem 2(b), p. 59]{MillerRoss93}  and the semigroup property of the operator $J_a$ we have 
\begin{equation}\label{eqn.RL2}
J_a^{k-\gamma}h_1=J_a^{k-\gamma}D(J_a^{1-\beta}h)= DJ_a^{k-\gamma}(J_a^{1-\beta}h)
= D (J_a^{k+1-(\beta+\gamma)}h).
\end{equation}
Now, if $\lceil\beta+\gamma\rceil= k+1$ then from \eqref{eqn.RL1}--\eqref{eqn.RL2} it follows that 
$$
DJ_a^{\lceil\beta+\gamma\rceil-(\beta+\gamma)}h=DJ_a^{k+1-(\beta+\gamma)}h=J_a^{k-\gamma}h_1\in C^{k}[a,b],
$$
implying $D^{k+1}J_a^{\lceil\beta+\gamma\rceil-(\beta+\gamma)}h= D^kJ_a^{k-\gamma}h_1$,
 hence
${}^{RL}\!D_{a}^{\beta+\gamma}h={}^{RL}\!D_{a}^{\gamma}h_1$. 
If 
$\lceil\beta+\gamma\rceil= k$ then from \eqref{eqn.RL1}--\eqref{eqn.RL2} it follows that 
$$
J_a^{\lceil\beta+\gamma\rceil-(\beta+\gamma)}h = D (J_a^{k+1-(\beta+\gamma)}h) = J_a^{k-\gamma}h_1\in C^k[a,b],
$$
implying $D^kJ_a^{\lceil\beta+\gamma\rceil-(\beta+\gamma)}h = D^kJ_a^{k-\gamma}h_1$,
hence
${}^{RL}\!D_{a}^{\beta+\gamma}h={}^{RL}\!D_{a}^{\gamma}h_1$.

The theorem is proved.
   \end{proof}

     \begin{theorem} \label{thm.RL2}
     Let   $\alpha >1$ be arbitrary,  and $f\in C[a,b]$ be continuously Riemann-Liouville $\alpha$-differentiable. Then
     
       (i) For those $0\leq\beta<\alpha$ such that $\alpha-\beta\in\N$, the function $f$ is continuously Riemann-Liouville $\beta$-differentiable with $f_1:={}^{RL}\!D_{a}^{\beta}f\in C[a,b]$. The function $f_1$ is continuously Riemann-Liouville 
     $(\alpha-\beta)$-differentiable, and $D^{\alpha-\beta}f_1={}^{RL}\!D_{a}^{\alpha}f$; thus 
     $D^{\alpha-\beta}{}^{RL}\!D_{a}^{\beta}f = 
       {}^{RL}\!D_{a}^{\alpha}f$.
       
       (ii) For any $0\leq\beta< \alpha- \lfloor \alpha\rfloor$, the function $f$ is continuously Riemann-Liouville $\beta$-differentiable with $f_1:={}^{RL}\!D_{a}^{\beta}f\in C[a,b]$. 
     Moreover, $({}^{RL}\!D_{a}^{\beta}f)(a) = 0$; the function $f_1$ is continuously Riemann-Liouville 
     $\gamma$-differentiable for any $0<\gamma\leq \alpha- \lfloor \alpha\rfloor-\beta$, and 
     ${}^{RL}\!D_{a}^{\gamma}f_1={}^{RL}\!D_{a}^{\beta+\gamma}f$, thus 
     ${}^{RL}\!D_{a}^{\gamma}{}^{RL}\!D_{a}^{\beta}f={}^{RL}\!D_{a}^{\beta+\gamma}f$.
   
   \end{theorem}
   \begin{proof}
(i) If $\alpha\in\N$ then $\beta\in\N$, hence (i) is clearly true due to the properties of the classical differentiation. Suppose that $\alpha\not\in\N$. 
Put $2\leq m:= \lceil \alpha\rceil\in\N$, $n:= \lceil \beta\rceil\in\N$. By assumption, $m > n$ and $m-\alpha=n-\beta$.
  Since $f\in C[a,b]$ is continuously Riemann-Liouville $\alpha$-differentiable we have $J_{a}^{m-\alpha}f\in C^m[a,b]$, hence $J_{a}^{n-\beta}f\in C^n[a,b]$. Consequently, $f$ is continuously Riemann-Liouville $\beta$-differentiable. Now using $J_{a}^{m-\alpha}f\in C^m[a,b]$, $\alpha-\beta=m-n$ and $f_1 = D^nJ_{a}^{n-\beta}f $ we get
  $$
  {}^{RL}\!D_{a}^{\alpha}f = D^mJ_{a}^{m-\alpha}f
  = D^{m-n}D^nJ_{a}^{m-\alpha}f = D^{m-n}D^nJ_{a}^{n-\beta}f
  = D^{\alpha-\beta}f_1.
  $$
  This means that $f_1$ is continuously Riemann-Liouville 
     $(\alpha-\beta)$-differentiable, and $D^{\alpha-\beta}f_1={}^{RL}\!D_{a}^{\alpha}f$.
     
(ii)   If $\alpha\in\N$ then $\alpha- \lfloor \alpha\rfloor=0$ and (ii) is obviously true. 
Suppose that $\alpha\not\in\N$. Put   $\alpha_1:= \alpha- \lfloor \alpha\rfloor$, $\alpha_2:= \beta+\gamma$. Then $0<\alpha_2< \alpha_1< 1$. By (i) above $f$ is continuously Riemann-Liouville $\alpha_1$-differentiable. By 
Theorem~\ref{thm.RL1}(i) the function $f$ is continuously Riemann-Liouville $\alpha_2$-differentiable. Applying Theorem~\ref{thm.RL1} to the function $f$, which is continuously Riemann-Liouville $\alpha_2$-differentiable with $0<\alpha_2<1$, we complete the proof of part (ii). 

The theorem is proved.
   \end{proof}

    \begin{proposition}\label{prp.RL3}
   (i)  Let  $\alpha\geq 1$, $0<\beta< \alpha$ be such that $\alpha-\beta\not\in\N$ and  $ \beta >  \alpha -\lfloor \alpha\rfloor$. Then there exists a function 
    $f\in C[a,b]$ such that $f$ is continuously Riemann-Liouville $\alpha$-differentiable with $({}^{RL}\!D_{a}^{\alpha}f)\in C[a,b]$ but $f$ is not continuously Riemann-Liouville $\beta$-differentiable. Moreover, if $\beta > (\alpha -\lfloor \alpha\rfloor)+1$ then $f$ is not Riemann-Liouville $\beta$-differentiable.
    
    (ii) Let  $\alpha>1$, $0\leq\beta<\alpha-\lfloor\alpha\rfloor < 1$ be arbitrary. 
    Then for any $\gamma > (\alpha-\lfloor\alpha\rfloor)- 
    \beta$, $(\beta+\gamma)-\alpha+\lfloor\alpha\rfloor\not\in\N$, there exists a continuous function $f\in C[a,b]$ such that $f$ is continuously Riemann-Liouville $\alpha$-differentiable and hence continuously Riemann-Liouville $\beta$-differentiable with derivative 
    $f_1:= {}^{RL}\!D_{a}^{\beta}f$ but the function $f_1\in C[a,b]$ is not continuously Riemann-Liouville $\gamma$-differentiable.
   \end{proposition}

\begin{proof}
(i) Put  $c:= \alpha-\lfloor \alpha\rfloor$. Then  $0\leq c<1$.
We choose $f:=(t-a)^c\in C[a,b]$.
By Diethelm~\cite[Example 2.1, p. 20]{Kai}, in case $c>0$ we have
   $$
   J_a^{\lceil \alpha\rceil-\alpha}f = J_a^{\lceil \alpha\rceil-\alpha}(t-a)^c= (t-a)\Gamma(c+1)\in C^\infty[a,b];
   $$
   and in case $c=0$ we have $J_a^{\lceil \alpha\rceil-\alpha}f \equiv 1\in C^\infty[a,b]$.
   Hence $f$ is continuously Riemann-Liouville $\alpha$-differentiable with $({}^{RL}\!D_{a}^{\alpha}f)\in C[a,b]$ for any $\alpha>0$.
   
   From assumption we have $\beta-c\not\in\N$, hence by Diethelm~\cite[Example 2.4, p. 28]{Kai} we have 
  $$
  {}^{RL}\!D_{a}^{\beta}f = 
   \frac{\Gamma(c+1)}{\Gamma(c+1-\beta)}(t-a)^{c-\beta}. 
   $$
   Furthermore, we have $\beta>c$, hence ${}^{RL}\!D_{a}^{\beta}f\not\in C[a,b]$. Consequently,  $f$ is not continuously Riemann-Liouville $\beta$-differentiable.
   
   Now, suppose that $\beta > (\alpha -\lfloor \alpha\rfloor)+1$, then $\beta-c>1$. In this case we have 
   ${}^{RL}\!D_{a}^{\beta}f\not\in L_1[a,b]$. Consequently,  $f$ is not Riemann-Liouville $\beta$-differentiable.
   
   (ii) If $\alpha\in\N$ then there is no $\beta$ and (ii) is true. 
   
   Suppose that $\alpha\not\in\N$.
  We use the same function $f$ as (i) above. By Diethelm~\cite[Example 2.4, p. 28]{Kai},  since  $(\beta+\gamma)-c\not\in\N$  we have
   $$
   {}^{RL}\!D_{a}^{\gamma} f_1={}^{RL}\!D_{a}^{\gamma}({}^{RL}\!D_{a}^{\beta}f)= 
   \frac{\Gamma(c+1)}{\Gamma(c+1-(\beta+\gamma))}(t-a)^{c-(\beta+\gamma)}.
   $$
   Since $c-(\beta+\gamma)<0$ we have ${}^{RL}\!D_{a}^{\gamma} f_1\not\in C[a,b]$ and (ii) is proved. 
    \end{proof}
   
    \begin{remark}\label{rem.RL}
   (i) In the proof of Theorem~\ref{thm.RL1} we need to use  essentially the assumption that $f\in C[a,b]$ is continuously Riemann-Liouville $\alpha$-differentiable. Note that this assumption implies that $f(a)=0$ (use  $\beta=0$ in  part (i) of the theorem). \\
  (ii) Theorem 2.13 in Diethelm~\cite{Kai} is similar to Theorem~\ref{thm.RL1}, but it requires a stronger smoothness condition on $f$, and due to the smoothness of $f$ it holds also for $\alpha>1$. As noted by Diethelm, without continuity assumption and zero condition $f(a)=0$ the semigroup property of the Riemann-Liouville differential operators  may not hold, 
   see Diethelm~\cite[Examples 2.6, 2.7, p. 30]{Kai}.\\
   (iii) Theorems \ref{thm.RL1}--\ref{thm.RL2} and Proposition~\ref{prp.RL3} provide us with a complete answer to the question of whether a continuous Riemann-Liouville $\alpha$-differentiable $f$ is always  continuous Riemann-Liouville $\beta$-differentiable for $\beta<\alpha$, and 
   provide us with a complete description of partial semigroup property of the Riemann-Liouville fractional differential operators.
      \end{remark}

        \subsection{Semigroup property of Caputo fractional differential operators}
   Since Caputo differentiation is defined via Riemann-Liouville differentiation it is not difficult to adapt the result of Subsection~\ref{subsec.RL} above to the case of the Caputo fractional differentiation. However, the Caputo differentiation has its distinctive feature what allows us to prove partial semigroup property of the Caputo fractional differential operators ${}^{C}\!D_{a}^{\alpha}f$, $0\leq\alpha\in\R$, for a bigger set of the parameter $\alpha$ than the Riemann-Liouville case.
   
   First we use Theorem~\ref{thm.RL1} to prove a similar result for Caputo fractional differentiation.
   
     \begin{proposition}\label{prp.Cap1}
   Let $0 <\beta<\alpha < 1$ be arbitrary, $f\in C[a,b]$ be continuously Caputo $\alpha$-differentiable. Then 
   
   (i) $f$ is continuously $\beta$-differentiable with continuous derivative ${}^{C}\!D_{a}^{\beta}f\in C[a,b]$. Moreover,  $({}^{C}\!D_{a}^{\beta}f)(a) = 0$.
   
 (ii)  The continuous fractional derivative $f_1:={}^{C}\!D_{a}^{\beta}f\in C[a,b]$ is continuously Caputo 
 $(\alpha-\beta)$-differentiable. Moreover, ${}^{C}\!D_{a}^{\alpha-\beta}f_1= {}^{C}\!D_{a}^{\alpha}f$, thus
 ${}^{C}\!D_{a}^{\alpha-\beta}{}^{C}\!D_{a}^{\beta}f = {}^{C}\!D_{a}^{\alpha}f$.  
   \end{proposition}
   \begin{proof}
(i) Put $g:=f-f(a)$, then $g\in C[a,b]$. Since $0<\beta<\alpha < 1$ by definition 
 ${}^{C}\!D_{a}^{\alpha}f := {}^{RL}\!D_{a}^{\alpha}g$.
Therefore $g$ is continuously Riemann-Liouville $\alpha$-differentiable. By Theorem~\ref{thm.RL1}, $g$ is continuously Riemann-Liouville $\beta$-differentiable, and $({}^{RL}\!D_{a}^{\beta}g)(a) = 0$. Consequently, $f$ is continuously Caputo $\beta$-differentiable, ${}^{C}\!D_{a}^{\beta}f={}^{RL}\!D_{a}^{\beta}g$
 and $({}^{C}\!D_{a}^{\beta}f)(a)=({}^{RL}\!D_{a}^{\beta}g)(a)=0$.

(ii) For $f_1={}^{C}\!D_{a}^{\beta}f\in C[a,b]$, by (i) above we have $f_1(a) = ({}^{C}\!D_{a}^{\beta}f)(a)=0$, and also $f_1={}^{RL}\!D_{a}^{\beta}g$. Since $g$ is continuously Riemann-Liouville $\alpha$-differentiable, by Theorem~\ref{thm.RL1}, $f_1$ is continuously Riemann-Liouville
 $(\alpha-\beta)$-differentiable and 
 ${}^{RL}\!D_{a}^{\alpha-\beta}f_1 = {}^{RL}\!D_{a}^{\alpha}g$.
 Consequently, since ${}^{RL}\!D_{a}^{\alpha-\beta}f_1= {}^{C}\!D_{a}^{\alpha-\beta}f_1$ because $0<\alpha-\beta<1$ and $f_1(a)=0$, taking account the equality   ${}^{RL}\!D_{a}^{\alpha}g={}^{C}\!D_{a}^{\alpha}f $ we get 
 ${}^{C}\!D_{a}^{\alpha-\beta}f_1= {}^{C}\!D_{a}^{\alpha}f$.
 \end{proof}

Unlike the Riemann-Liouville case where the semigroup property is lost as soon as $\alpha>1$ as shown in Proposition~\ref{prp.RL3}, in the Caputo case we still have some limited semigroup property of the operator of  Caputo continuous fractional differentiation for $\alpha>1$ as asserted in the theorems below. To derive the theorem we will use the results by Vainikko~\cite{Vainikko2016}. 

    Recall that for a smooth function $f\in C^m[a,b]$, $m\in\N$, we denote by $T_m[f;a]:= \sum_{k=0}^m \frac{(D^kf)(a)}{k!}(t-a)^k$ its Taylor polynomial of order $m$ centered at $a$. 
   For $m\in\N$ we put
  $$
  C_0^m[a,b]:= \{ v\in C^m[a,b]: v^k(0)=0, k=0,1,\ldots, m-1\}.
  $$
   It is easily seen that  for $m\in\N$ the range of the Riemann-Liouville fractional integral operator $J_a^m$, in the continuous setting, is
   $$
   J_a^m C[a,b] = \{ v\in C^m[a,b]: v^k(0)=0, k=0,1,\ldots, m-1\} = C_0^m[a,b].
   $$
We denote by $D_0^\alpha$, $0\leq \alpha\in\R$, the inverse of the Riemann-Liouville fractional integral operator $J_a^\alpha$, i.e. 
$$
D_0^\alpha v = (J_a^\alpha)^{-1} v, \quad\hbox{for}\quad v\in J_a^\alpha C[a,b];
$$
i.e., $D_0^\alpha$ is defined on $J_a^\alpha C[a,b]$, and for $v\in J_a^\alpha C[a,b]$, hence exists $u\in C[a,b]$ such that $J_a^\alpha u=v$, we set $D_0^\alpha v=u$. Clearly the range of $D_0^\alpha$ is a subset of $C[a,b]$.
Using the semigroup property of $J_a^\alpha$ one can prove that the domain of $D_0^\alpha$ is a subset of $C_0^{\lfloor\alpha\rfloor}[a,b] \subset C[a,b]$; and the operators $D_0^\alpha$, $0\leq \alpha\in\R$, has semigroup property with respect to the parameter $\alpha$, like the classical differential operators $D^n$, $n\in\N$, does. For more details on the operator $D_0^\alpha$ and its properties as well as its relation to Caputo and Riemann-Liouville derivatives we refer the reader to Vainikko~\cite{Vainikko2016}.

       \begin{theorem}\label{thm.Cap2}
  Let $\alpha>0$,  $\lceil \alpha\rceil -1\leq\beta< \alpha$ be arbitrary, $f\in C^{\lceil \alpha\rceil -1}[a,b]$ be continuously Caputo $\alpha$-differentiable. Then 
  
   (i) $f$ is continuously $\beta$-differentiable with continuous derivative ${}^{C}\!D_{a}^{\beta}f\in C[a,b]$. Moreover, if $\beta\not\in\N$ then $({}^{C}\!D_{a}^{\beta}f)(a) = 0$.
   
 (ii)  The continuous fractional derivative $f_1:={}^{C}\!D_{a}^{\beta}f\in C[a,b]$ is continuously Caputo 
 $(\alpha-\beta)$-differentiable. Moreover, ${}^{C}\!D_{a}^{\alpha-\beta}f_1= {}^{C}\!D_{a}^{\alpha}f$, thus
 ${}^{C}\!D_{a}^{\alpha-\beta}{}^{C}\!D_{a}^{\beta}f = {}^{C}\!D_{a}^{\alpha}f$.  
   \end{theorem}
   \begin{proof}
(i)   Put $m:= \lceil \alpha\rceil$. Then $m\geq 1$,  $f\in C^{m-1}[a,b]$. 
  Put $g:=T_{m-1}[f;a] =\sum_{k=0}^{m-1}\frac{(D^kf)(a)}{k!}(t-a)^k$. 
        
 1. First we consider the case $\alpha\not\in\N$.  By Vainikko~\cite[Proposition 5.1]{Vainikko2016}, $f$ is continuously Caputo $\alpha$-differentiable if and only if $f-g$ is $D_0^\alpha$-differentiable. This implies that $f-g$ is $D_0^\beta$-differentiable because $0<\beta<\alpha$ and the operators $D_0^r$, $r>0$, has semigroup property with respect to the parameter $r$.  Then, again by virtue of  Vainikko~\cite[Proposition 5.1]{Vainikko2016} we get that $f$ is continuously Caputo $\beta$-differentiable.  
By Vainikko~\cite[Theorem 5.2]{Vainikko2016}, since $f\in C^{m-1}[a,b]$ is continuously Caputo $\alpha$-differentiable, there exists a finite limit
\begin{equation}\label{eqn.Cap1}
\lim_{t\to a} (t-a)^{m-1-\alpha}((D^{m-1}f)(t)- (D^{m-1}f)(a))=:\gamma_m,
\end{equation}
and $({}^{C}\!D_{a}^{\alpha}f)(a)= \Gamma(\alpha+2-m)\gamma_m$. Similarly, if $\beta\not\in\N$ there exists a finite limit
\begin{equation}\label{eqn.Cap2}
\lim_{t\to a} (t-a)^{m-1-\beta}((D^{m-1}f)(t)- (D^{m-1}f)(a))=:\gamma_m',
\end{equation}
and $({}^{C}\!D_{a}^{\beta}f)(a)= \Gamma(\beta+2-m)\gamma_m'$. Since $\alpha-\beta>0$, from \eqref{eqn.Cap1}-\eqref{eqn.Cap2} we get
\begin{eqnarray*}
\gamma_m' &=&   \lim_{t\to a} (t-a)^{m-1-\beta}((D^{m-1}f)(t)- (D^{m-1}f)(a))\\
 &=&  \lim_{t\to a}(t-a)^{\alpha-\beta} (t-a)^{m-1-\alpha}((D^{m-1}f)(t)- (D^{m-1}f)(a))\\
 &=&0\times \gamma_m=0.
   \end{eqnarray*}
   Thus $({}^{C}\!D_{a}^{\beta}f)(a)=0$ if $\beta\not\in\N$. 
   
   If $\beta\in\N$ then $\beta= m-1$ and clearly $f\in C^\beta[a,b]$, hence is continuously $\beta$-differentiable.
   
   2. Now we consider the case $\alpha\in\N$. We have $\alpha=\lceil\alpha\rceil=m$ and $m-1\leq\beta<m$. Since $f$ is continuously Caputo $\alpha$-differentiable,  $f\in C^m[a,b]$. Therefore, by Diethelm~\cite[Lemma 3.11, p. 56]{Kai}, $f$ is continuously Caputo $\beta$-differentiable with continuous derivative $({}^{C}\!D_{a}^{\beta}f)\in C[a,b]$ and, moreover,  $({}^{C}\!D_{a}^{\beta}f)(a) = 0$ in case $\beta\not\in\N$.
   If $\beta\in\N$ then $\beta = m-1$ and clearly $f\in C^\beta[a,b]$, hence is continuously $\beta$-differentiable.
Part (i) of the theorem is proved.
 
 (ii) 1. First we consider the case $\alpha\not\in\N$ and $\beta\not\in\N$, thus $\lceil\alpha\rceil=m$ and  $\beta>m-1$. 
 Put ${\hat f}:=f-g$, where  $g = T_{m-1}[f;a]= \sum_{k=0}^{m-1}\frac{(D^kf)(a)}{k!}(t-a)^k$.
 Since $f$ is continuously Caputo $\alpha$-differentiable, by 
Vainikko~\cite[Proposition 5.1]{Vainikko2016} the function ${\hat f}=f-g$ is $D_0^\alpha$-differentiable and $D_0^{\alpha}{\hat f} = {}^{C}\!D_{a}^{\alpha}f$.
Using semigroup property of the operators $D_0^r$, $r>0$, with respect to the parameter $r$, since $0 < \beta <\alpha$ the function $\hat f$ is $D_0^{\beta}$-differentiable and $D_0^\beta{\hat f}$ is $D_0^{\alpha-\beta}$-differentiable,  furthermore
\begin{equation}\label{eqn.Cap3}
D_0^{\alpha-\beta}(D_0^\beta{\hat f})= D_0^{\alpha}{\hat f} = {}^{C}\!D_{a}^{\alpha}f.
\end{equation}

Due to part (i) above $f$ is continuously Caputo $\beta$-differentiable, hence 
the function ${\hat f}=f-g$ is $D_0^\beta$-differentiable with
derivative $D_0^\beta(f-g)\in C[a,b]$ and ${}^{C}\!D_{a}^{\beta}f = D_0^\beta(f-g)$, that is 
\begin{equation}\label{eqn.Cap4}
f_1=D_0^\beta{\hat f},
\end{equation}
and by \eqref{eqn.Cap3} the  continuous fractional derivative $f_1={}^{C}\!D_{a}^{\beta}f$ is continuously  
 $D_0^{\alpha-\beta}$-differentiable. 
Since $0<\alpha-\beta<1$, by Vainikko~\cite[Proposition 5.1]{Vainikko2016} we have 
\begin{equation}\label{eqn.Cap5}
D_0^{\alpha-\beta}(D_0^\beta{\hat f}) = {}^{C}\!D_a^{\alpha-\beta}((D_0^\beta{\hat f}) - (D_0^\beta{\hat f})(a)).
\end{equation}

Since  $f$ is continuously Caputo $\beta$-differentiable the function ${\hat f}=f-g$ is continuously Caputo $\beta$-differentiable, because $g$ is a polynomial of degree not greater than $\lceil\beta\rceil -1$ hence  ${}^{C}\!D_{a}^{\beta}g =0$. 
Therefore, by Vainikko~\cite[Proposition 5.1]{Vainikko2016},  the function ${\hat f}- T_{m-1}[{\hat f};a]$ has continuous derivative $D_0^\beta({\hat f}- T_{m-1}[{\hat f};a])\in C[a,b]$ and ${}^{C}\!D_{a}^{\beta}{\hat f} = D_0^\beta({\hat f}- T_{m-1}[{\hat f};a])$. Moreover, since $\beta\not\in\N$, 
${}^{C}\!D_{a}^{\beta}{\hat f} (a)=0$. Since $g = \sum_{k=0}^{m-1}\frac{(D^kf)(a)}{k!}(t-a)^k$ we have $T_{m -1}[{\hat f};a] = 0$, hence  
\begin{equation}\label{eqn.Cap6}
{}^{C}\!D_{a}^{\beta}{\hat f} = D_0^\beta{\hat f}\quad\hbox{and}\quad
D_0^\beta{\hat f}(a) ={}^{C}\!D_{a}^{\beta}{\hat f} (a)=0.
\end{equation}

From \eqref{eqn.Cap4}--\eqref{eqn.Cap6} it follows that the function $f_1$ is continuously Caputo $(\alpha-\beta)$-differentiable; and moreover, by virtue of \eqref{eqn.Cap3} we have
 $$
 {}^{C}\!D_{a}^{\alpha}f = D_0^{\alpha-\beta}(D_0^\beta{\hat f})
 = {}^{C}\!D_a^{\alpha-\beta}f_1.
 $$

2. The case $\alpha\not\in \N$ and $\beta\in\N$, thus $\beta=m-1<\alpha<m$. Put $f_2:= D^{m-1}f$, then we have $f_2\in C[a,b]$ and $f_2 = {}^{C}\!D_{a}^{\beta}f$. By Vainikko~\cite[Theorem 5.2]{Vainikko2016} the continuous Caputo derivative of order $\alpha$ of $f$ is given by 
\begin{eqnarray*}
{}^{C}\!D_{a}^{\alpha}f (t) &=& \frac{1}{\Gamma(m-\alpha)}\Big( (t-a)^{m-1-\alpha}(f_2(t)-f_2(a))\\
 &&+\; (\alpha-m+1)\int_a^t (t-s)^{m-\alpha-2}(f_2(t)-f_2(s)) ds\Big), \quad a\leq t\leq b.
 \end{eqnarray*}
We have $0<\alpha-\beta=\alpha-m+1<1$ and $\lceil\alpha-\beta\rceil=1$.  Applying Vainikko~\cite[Theorem 5.2]{Vainikko2016} to the function $f_2$ we get that $f_2$ is continuously Caputo $(\alpha-\beta)$-differentiable and the Caputo derivative of order  $(\alpha-\beta)$ of  $f_2$ is given by
\begin{eqnarray*}
{}^{C}\!D_{a}^{\alpha-\beta}f _2(t) &=& 
\frac{1}{\Gamma(m-\alpha)}\Big( (t-a)^{m-1-\alpha}(f_2(t)-f_2(a))\\
 &&+\; (\alpha-m+1)\int_a^t (t-s)^{m-\alpha-2}(f_2(t)-f_2(s)) ds\Big), \quad a\leq t\leq b.
 \end{eqnarray*}
Consequently, ${}^{C}\!D_{a}^{\alpha-\beta}{}^{C}\!D_{a}^{\beta}f ={}^{C}\!D_{a}^{\alpha-\beta}f _2
= {}^{C}\!D_{a}^{\alpha}f$.

3. The case $\alpha\in \N$ and $\beta\not\in\N$, thus $\alpha=m$ and $\lceil\beta\rceil=m$. In this case $f\in C^m[a,b]$, and by  Diethelm~\cite[Lemma 3.13, p. 56]{Kai} we have 
${}^{C}\!D_{a}^{\alpha-\beta}{}^{C}\!D_{a}^{\beta}f = {}^{C}\!D_{a}^{\alpha}f$. 

4. The case $\alpha\in\N$ and $\beta\in\N$, thus $\alpha=m$ and $\beta=m-1$. In this case the Caputo differentiation coincide with the classical differentiation. Since $f$ is continuously Caputo $\alpha$-differentiable we have $f\in C^m[a,b]$ and clearly 
${}^{C}\!D_{a}^{\alpha-\beta}{}^{C}\!D_{a}^{\beta}f = {}^{C}\!D_{a}^{\alpha}f$.

Part (ii) is proved. 
The theorem is proved completely.
    \end{proof}

     The following theorem asserts the partial semigroup property of the Caputo fractional differential operators. This theorem includes Theorem~\ref{thm.Cap2} as a particular case, whereas Theorem~\ref{thm.Cap2} is a generalization of Proposition~\ref{prp.Cap1}.
    \begin{theorem}[Partial semigroup property of Caputo fractional differential operators]\label{thm.Cap3}
   I. Let  $\alpha>0$  be arbitrary and $f\in C^{\lceil \alpha\rceil -1}[a,b]$ be continuously Caputo $\alpha$-differentiable. Then
    
 (i)  For any $0<\beta<\alpha$,  $f$ is continuously Caputo  $\beta$-differentiable with continuous derivative $f_1:={}^{C}\!D_{a}^{\beta}f\in C[a,b]$.  Furthermore, if $\beta\not\in\N$ then $({}^{C}\!D_{a}^{\beta}f)(a) = 0$.
 
 (ii)  For any  $0 < \gamma\leq 1$ such that $\beta+\gamma\leq\alpha$ and $\lceil\beta+\gamma\rceil- \lfloor\beta\rfloor=1$, the continuously Caputo $\beta$-derivative $f_1$ provided by (i) above is continuously Caputo $\gamma$-differentiable, and 
 ${}^{C}\!D_{a}^{\gamma}f_1= {}^{C}\!D_{a}^{\beta+\gamma}f$; thus
 ${}^{C}\!D_{a}^{\gamma}{}^{C}\!D_{a}^{\beta}f = {}^{C}\!D_{a}^{\beta+\gamma}f$.
 
  II. Conversely, let $\beta>0$, $0 < \gamma\leq 1$ be such that  $\lceil\beta+\gamma\rceil- \lfloor\beta\rfloor=1$. Assume that a continuous function $h\in C^{\lceil \beta\rceil -1}[a,b]$ is continuously Caputo $\beta$-differentiable such that ${}^{C}\!D_{a}^\beta h=:h_1$ is continuously Caputo $\gamma$-differentiable, and assume additionally that $h_1(a)=0$ if $\beta\not\in\N$. Then $h$ is continuously Caputo $(\beta+\gamma)$-differentiable, and 
 ${}^{C}\!D_{a}^{\gamma}h_1= {}^{C}\!D_{a}^{\beta+\gamma}h$, thus 
 ${}^{C}\!D_{a}^{\gamma}({}^{C}\!D_{a}^{\beta}h)= {}^{C}\!D_{a}^{\beta+\gamma}h$.
    \end{theorem}
  
    \begin{proof}
   I. (i) 
     If $\lceil \alpha\rceil -1\leq\beta< \alpha$ then we are done by using Theorem~\ref{thm.Cap2}(i).
    If $\beta<\lceil \alpha\rceil-1$, then $\lfloor\beta\rfloor + 1\leq \lceil \alpha\rceil -1$, hence $f\in C^{\lfloor\beta\rfloor + 1}[a,b]$. Consequently $f$ is continuously Caputo $(\lfloor\beta\rfloor + 1)$-differentiable. Therefore, by Theorem~\ref{thm.Cap2} $f$ is continuously Caputo  $\beta$-differentiable with continuous derivative $f_1:={}^{C}\!D_{a}^{\beta}f\in C[a,b]$; and moreover, if $\beta\not\in\N$ then $({}^{C}\!D_{a}^{\beta}f)(a) = 0$.
    
    (ii) By (i) above, since $\beta+\gamma\leq\alpha$ the function $f$ is continuously Caputo
    $(\beta+\gamma)$-differentiable. 
    
    Now, since $\gamma>0$ and $\lceil\beta+\gamma\rceil- \lfloor\beta\rfloor=1$ we have $\lceil\beta+\gamma\rceil-1\leq \beta<\beta+\gamma$. Therefore, by Theorem~\ref{thm.Cap2} the continuously Caputo $\beta$-derivative $f_1$ provided by (i) above is continuously Caputo $\gamma$-differentiable, and 
 ${}^{C}\!D_{a}^{\gamma}f_1= {}^{C}\!D_{a}^{\beta+\gamma}f$; thus
 ${}^{C}\!D_{a}^{\gamma}{}^{C}\!D_{a}^{\beta}f = {}^{C}\!D_{a}^{\beta+\gamma}f$.
 
 II. {\em Case 1: $\beta\not\in\N$.}  We put ${\hat h}:= h- T_{\lceil\beta\rceil-1}[h;a]$. 
Since $h$ is continuously Caputo $\beta$-differentiable, by 
Vainikko~\cite[Proposition 5.1]{Vainikko2016} the function ${\hat h}$ is $D_0^\beta$-differentiable and $D_0^{\beta}{\hat h} = {}^{C}\!D_{a}^{\beta}h=h_1$. Note that, in this case $0<\gamma<1$.
 Therefore, since by assumption $h_1$ is continuously  Caputo $\gamma$-differentiable  and $h_1(a)=0$, the function $h_1$ is  $D_0^\gamma$-differentiable and 
${}^{C}\!D_{a}^{\gamma}h_1  = D_0^\gamma h_1$. Consequently, by the semigroup property of the 
operators $D_0^r$, $r>0$, with respect to the parameter $r$,  the function ${\hat h}$ is 
$D_0^{\beta+\gamma}$-differentiable and we have
$$
 D_0^{\beta+\gamma}{\hat h} = D_0^{\gamma}D_0^{\beta}{\hat h} = D_0^{\gamma}h_1=
 {}^{C}\!D_{a}^{\gamma}h_1. 
$$
Since $\lceil\beta\rceil=\lceil\beta+\gamma\rceil$ and ${\hat h}$ is $D_0^{\beta+\gamma}$-differentiable, in case $\beta+\gamma\not\in\N$ by Vainikko~\cite[Proposition 5.1]{Vainikko2016} the function $h$ is continuously Caputo $(\beta+\gamma)$-differentiable, and ${}^{C}\!D_{a}^{\beta+\gamma}h = D_0^{\beta+\gamma}{\hat h}$, hence
${}^{C}\!D_{a}^{\beta+\gamma}h ={}^{C}\!D_{a}^{\gamma}h_1$. Otherwise, if 
$\beta+\gamma\in\N$ then,  since $T_{\lceil\beta\rceil-1}[h;a]$ is a polynomial of degree not greater than $\beta+\gamma-1$, we have
$$
D_0^{\beta+\gamma}{\hat h}= D^{\beta+\gamma}{\hat h} = D^{\beta+\gamma}(h - T_{\lceil\beta\rceil-1}[h;a])
= D^{\beta+\gamma} h = {}^{C}\!D_{a}^{\beta+\gamma}h,
$$
hence ${}^{C}\!D_{a}^{\beta+\gamma}h= {}^{C}\!D_{a}^{\gamma}h_1$.

{\em  Case 2: $\beta\in\N$, $\beta+\gamma\not\in\N$.} In this case we have $\lceil\beta+\gamma\rceil=\beta+1$, $\gamma<1$. The function $h\in C^\beta[a,b]$ and $h_1={}^{C}\!D_{a}^{\beta}h = D^\beta h\in C[a,b]$ because $h$ is continuously Caputo $\beta$-differentiable. By Vainikko~\cite[Theorem 5.2]{Vainikko2016}, since $h_1$ is continuously Caputo $\gamma$-differentiable, $0<\gamma<1$, we have the finite limit $\lim_{t\to a}(t-a)^{-\gamma}(h_1(t)-h_1(a))=: c_1$ exists and
$$
\sup_{a< t\leq b}\Big|\int_{\theta t}^t (t-s)^{-\gamma-1} (h_1(t)-h_1(s)) ds\Big| \to 0\quad\hbox{as}\quad \theta\uparrow 1.
$$
Now apply Vainikko~\cite[Theorem 5.2]{Vainikko2016} to the function $h\in C^\beta[a,b]$ with $D^\beta h = h_1$ we get that the function $h$ is continuously Caputo $(\beta+\gamma)$-differentiable and 
${}^{C}\!D_{a}^{\beta+\gamma}h(a)= \Gamma(\gamma+1)c_1={}^{C}\!D_{a}^{\gamma}h_1(a)$. Moreover, for $a<t\leq b$ we have
\begin{eqnarray*}
{}^{C}\!D_{a}^{\beta+\gamma}h (t) &=& \frac{1}{\Gamma(1-\gamma)}\Big( (t-a)^{-\gamma}(h_1(t)-h_1(a))\\
 &&\hspace*{2cm} +\; \gamma\int_a^t (t-s)^{-\gamma-1}(h_1(t)-h_1(s)) ds\Big)\\
 &=& {}^{C}\!D_{a}^{\gamma}h_1 (t). 
 \end{eqnarray*}
Thus, 
${}^{C}\!D_{a}^{\beta+\gamma}h ={}^{C}\!D_{a}^{\gamma}h_1$.

{\em Case 3: $\beta\in\N$, $\beta+\gamma\in\N$.} Our assertion is true and this case due to the rule of classical integer order differentiation.
          \end{proof}
    
       \begin{theorem}\label{thm.Cap4}
    Let $\alpha>0$  be arbitrary and $f\in C^{\lceil \alpha\rceil -1}[a,b]$ be continuously Caputo $\alpha$-differentiable. Then
  for any  $l\in\N$, $0\leq l\leq  \lceil\alpha\rceil-1$ the function $f$ is continuously Caputo $l$-differentiable with $f_1:=D^{l}f\in C[a,b]$, and the derivative $f_1$ is continuously Caputo $(\alpha-l)$-differentiable. Moreover, 
  ${}^{C}\!D_{a}^{\alpha-l}f_1= {}^{C}\!D_{a}^{\alpha}f$; thus
 ${}^{C}\!D_{a}^{\alpha-l}D^{l}f = {}^{C}\!D_{a}^{\alpha}f$.
    \end{theorem}
    \begin{proof} In case $\alpha\in\N$ we are back to the classical integer order differentiation and the theorem trivially holds. Assume that $\alpha\not\in\N$.
    
   Put $m:=  \lceil\alpha\rceil$, $g:= T_{m-1}[f;a]$, $f_2:=D^{m-1}f$. Since $f\in C^{m-1}[a,b]$ is continuously Caputo $\alpha$-differentiable, by
    Vainikko~\cite[Theorem 5.2]{Vainikko2016} the continuous Caputo derivative of order $\alpha$ of $f$ is given by 
\begin{eqnarray*}
{}^{C}\!D_{a}^{\alpha}f (t) &=& \frac{1}{\Gamma(m-\alpha)}\Big( (t-a)^{m-1-\alpha}(f_2(t)-f_2(a))\\
 &&+\; (\alpha-m+1)\int_a^t (t-s)^{m-\alpha-2}(f_2(t)-f_2(s)) ds\Big), \quad a\leq t\leq b.
 \end{eqnarray*}
Clearly, $f_2=D^{m-1-l}f_1$, hence  by
    Vainikko~\cite[Theorem 5.2]{Vainikko2016} the function $f_1$ is continuously Caputo $(\alpha-l)$-differentiable, and its continuous Caputo derivative of order $(\alpha-l)$ is given by
\begin{eqnarray*}
{}^{C}\!D_{a}^{\alpha-\beta}f _1(t) &=& 
\frac{1}{\Gamma(m-\alpha)}\Big( (t-a)^{m-1-\alpha}(f_2(t)-f_2(a))\\
 &&+\; (\alpha-m+1)\int_a^t (t-s)^{m-\alpha-2}(f_2(t)-f_2(s)) ds\Big), \quad a\leq t\leq b.
 \end{eqnarray*}
 Consequently,  ${}^{C}\!D_{a}^{\alpha-l}f_1= {}^{C}\!D_{a}^{\alpha}f$; thus
 ${}^{C}\!D_{a}^{\alpha-l}D^{l}f = {}^{C}\!D_{a}^{\alpha}f$.          
 \end{proof}

    \begin{proposition}\label{prp.Cap5}
    (i) For any $0<\alpha, \beta\in\R$ such that $\beta\not\in\N$, $\lceil\beta\rceil< \alpha$ there exits a continuous function $f\in C[a,b]$ such that $f$ is continuously Caputo $\alpha$-differentiable hence 
    continuously Caputo $\beta$-differentiable, but the continuous Caputo derivative ${}^{C}\!D_{a}^{\beta}f$ is not continuously Caputo $(\alpha-\beta)$-differentiable.
    
    (ii) For any $0<\alpha, \beta\in\R$ such that $\beta\not\in\N$, $\lceil\beta\rceil+1< \alpha$  there exits a continuous function $f\in C[a,b]$ such that $f$ is continuously Caputo $\alpha$-differentiable hence 
    continuously Caputo $\beta$-differentiable, but the continuous Caputo derivative ${}^{C}\!D_{a}^{\beta}f$ is not Caputo $(\alpha-\beta)$-differentiable.
    
    (iii) For any $\beta,\gamma>0$, $\beta\not\in\N$, and $\gamma\not\in\N$, there exists a continuous function $h\in C^{\lceil \beta\rceil -1}[a,b]$ which is continuously Caputo $\beta$-differentiable  such that  ${}^{C}\!D_{a}^\beta h$ is continuously Caputo $\gamma$-differentiable, but
    ${}^{C}\!D_{a}^{\gamma}({}^{C}\!D_{a}^{\beta}h)\not= {}^{C}\!D_{a}^{\beta+\gamma}h$.
    \end{proposition}
    
    \begin{proof}
    Put  $m:=\lceil\alpha\rceil$, $n:= \lceil\beta\rceil$.  Choose $f:= (t-a)^n$ for both parts (i)-(ii) of the proposition.
 By assumption, $m> n$ and $n-1<\beta<n$.   Since $f\in C^\infty[a,b]$ the function $f$ is continuously Caputo $\beta$-differentiable as well as 
    continuously Caputo $\alpha$-differentiable. A simple calculation gives
    \begin{eqnarray*}
    {}^{C}\!D_{a}^{\alpha}f &=& J_a^{m-\alpha} D^{m}f= 0,\\
     f_1:=  {}^{C}\!D_{a}^{\beta}f  &=& J_a^{n-\beta} D^{n}f=J_a^{n-\beta}{n!}= \frac{\Gamma(n+1)}{\Gamma(n-\beta+1)} (t-a)^{n-\beta}.
    \end{eqnarray*}
     Consequently, we have
    \begin{eqnarray*}
    {}^{C}\!D_{a}^{\alpha-\beta}f_1 &=& {}^{C}\!D_{a}^{\alpha-\beta}\frac{\Gamma(n+1)}{\Gamma(n-\beta+1)} (t-a)^{n-\beta}\\
    &=& \frac{\Gamma(n+1)}{\Gamma(n+1-\alpha)}(t-a)^{n-\alpha}.
    \end{eqnarray*} 
    
 (i)    Since by assumption $\alpha>\lceil\beta\rceil=n$ we have $n-\alpha<0$. Therefore $(t-a)^{n-\alpha}\not\in C[a,b]$, hence $f_1= {}^{C}\!D_{a}^{\beta}f$ is  not continuously Caputo $(\alpha-\beta)$-differentiable. Furthermore,  in case $-1<n-\alpha<0$ we have $(t-a)^{n-\alpha}\in L_1[a,b]$ hence $f_1$ is Caputo  $(\alpha-\beta)$-differentiable, but ${}^{C}\!D_{a}^{\alpha}f  \not= {}^{C}\!D_{a}^{\alpha-\beta}f_1$.
    
    (ii) Since by assumption $\alpha>\lceil\beta\rceil+1$ we have $n-\alpha<-1$. Therefore $(t-a)^{n-\alpha}\not\in L_1[a,b]$, hence ${}^{C}\!D_{a}^{\beta}f$ is  not Caputo $(\alpha-\beta)$-differentiable.
    
    (iii) Take $h:= (t-a)^\beta$ then $h$ is continuously Caputo $\beta$-differentiable with derivative $h_1={}^{C}\!D_{a}^{\beta}h = \Gamma(\beta+1)$. The function $h_1$ is continuously Caputo $\gamma$-differentiable, and  ${}^{C}\!D_{a}^{\gamma}h_1=0$. On the other hand, 
    $$
    {}^{C}\!D_{a}^{\beta+\gamma}h = {}^{C}\!D_{a}^{\beta+\gamma}(t-a)^\beta =
    \frac{\Gamma(\beta+1)}{\Gamma(1-\gamma)}(t-a)^{-\gamma}.
    $$
    Therefore, ${}^{C}\!D_{a}^{\gamma}({}^{C}\!D_{a}^{\beta}h)= {}^{C}\!D_{a}^{\gamma}h_1
    \not= {}^{C}\!D_{a}^{\beta+\gamma}h$.
    \end{proof}
       
   \begin{remark}\label{rem.Cap}
   (i) The assumption that $f$ is continuous Caputo $\alpha$-differentiable in  Proposition~\ref{prp.Cap1} does not imply that $f(a)=0$ as in the Riemann-Liouville case (see Remark~\ref{rem.RL}(i)).\\ 
   (ii) Theorem 3.13 in Diethelm~\cite{Kai} is similar to Theorem~\ref{thm.Cap4}, but it requires a stronger smoothness condition on $f$. \\
   (iii) Theorem \ref{thm.Cap2}--\ref{thm.Cap4} and Propositions \ref{prp.Cap1}, \ref{prp.Cap5} provide us with a complete answer to the question of whether a continuous Caputo $\alpha$-differentiable $f$ is always  continuous Caputo $\beta$-differentiable for $\beta<\alpha$, and provide us a complete description of a (partial) semigroup property of the Caputo fractional differential operators.
       \end{remark}
      
      \section{Applications}\label{sec.applications}
      
      In this section we present some applications of the semigroup property of  Caputo fractional differential operators derived in the preceding section. Namely we will present our results on reduction of a multi-term fractional differential equation to a single-term or to a multi-order fractional differential system, and a theorem on existence and uniqueness of solution to a general multi-term Caputo fractional differential system. 
      
      In this section, speaking of a solution of Caputo fractional differential equation we mean a function which is continuously Caputo differentiable of the desired order and satisfies the equation in consideration; and for the vector case we do the differentiation and computation component-wise. Note that this definition is widely used in the literature \cite{Kai}, and it allow us to show that the solution of the Caputo fractional differential equation with continuous coefficients is equivalent to the solution of the corresponding Volterra integral equation (see 
      Diethelm, Siengmund and Tuan~\cite[Definition 2.1]{DieSieTuan2017}, Diethelm~\cite[Chapter 6]{Kai})
      
      \subsection{Reduction of a multi-term system to a single-term or multi-order system of Caputo fractional differential equations}
      
  In this subsection we deal with the problem of reduction of multi-term Caputo fractional differential equation to  a  system of single-term or multi-order Caputo fractional differential equation.
  
  In certain cases we have to deal with Caputo fractional differential equations of order higher than 1 or equation of many differential term of different orders, for example in the problem of fractional relaxation and oscillation (see Gorenflo et all~\cite[Section 8.1.2]{Gorenflo2014}), or the Bagley-Torvik equation (see Diethelm~\cite[Chapter 8]{Kai}). The results in this paper provide us with a tool to deal with such kind of equations by reducing a multi-term equation to a single or multi-order fractional differential system. 
  This is achieved by the method of change of variable.  This approach is simple and similar to the classical theory of ordinary differential equations, and is treated in the book by Diethelm~\cite[Chapter 8]{Kai}. The main key here is the connection of the Caputo fractional derivatives  of different orders---the semigroup property---provided by Theorem \ref{thm.Cap3}.  

Since the semigroup property derived in Theorem \ref{thm.Cap3} is partial, i.e. true for only limited set of parameter of differentiation, the reduction we are able to do is not simple as its counterpart in the classical theory of ordinary differential equations. 

First, we present a reduction theorem in the case of rational orders of Caputo differentiation.
  Let us emphasize that while our Theorem~\ref{thm.reduce} below looks like Theorem 8.1 of Diethelm \cite{Kai} the important difference is the fact that using our result on semigroup property in Section~\ref{sec.RelDifOrders} we are able to drop the assumption on high smoothness of the solution, thus our theorem provides a natural equivalence of the systems under consideration.
  
   \begin{theorem}[cf. Theorem 8.1 of Diethelm \cite{Kai}] \label{thm.reduce}
    Consider the equation
    \begin{equation}\label{eqn.reduction1}
    {}^{C}\!D_{a}^{\alpha_k}x(t) = f(t,x(t), {}^{C}\!D_{a}^{\alpha_1}x(t),
    {}^{C}\!D_{a}^{\alpha_2}x(t), \ldots, {}^{C}\!D_{a}^{\alpha_{k-1}}x(t)),
    \end{equation}
  $t\in [a,b]$,  subject to the initial conditions
    \begin{equation}\label{eqn.reduction2}
    (D^jx)(a) = x_a^{(j)},
    \quad j=0,1,\ldots \lceil \alpha_k\rceil -1,
    \end{equation}
    where $\alpha_k>\alpha_{k-1}>\cdots >0$, $\alpha_j-\alpha_{j-1}\leq 1$ for all $j=2,3,\ldots,k$ and $0<\alpha_1\leq 1$, $f: [a,b]\times \R^k \to\R$ is a continuous function. Assume that $\alpha_j\in\Q$ for all $j=1,2,\ldots, k$, define $M$ to be the least common multiple of the denominators of $\alpha_1,\alpha_2,\ldots,\alpha_k$ and set $\gamma=1/M$ and $N=M\alpha_k$. Then this initial value problem is equivalent to the system of equations
    \begin{equation}\label{eqn.reduction3}
    \begin{array}{rcl}
    {}^{C}\!D_{a}^{\gamma}x_0(t) &=&x_1(t),\\
    {}^{C}\!D_{a}^{\gamma}x_1(t) &=&x_2(t),\\
    &\vdots&\\
    {}^{C}\!D_{a}^{\gamma}x_{N-2}(t) &=&x_{N-1}(t),\\
    {}^{C}\!D_{a}^{\gamma}x_{N-1}(t) &=&
    f(t,x_0(t), {}^{C}\!D_{a}^{\alpha_1}x_{\alpha_1/\gamma}(t),
     \ldots, {}^{C}\!D_{a}^{\alpha_{k-1}}x_{\alpha_{k-1}/\gamma}(t)),
    \end{array}
    \end{equation}
    together with the initial conditions
     \begin{equation}\label{eqn.reduction4}
     x_j(a)=\left\{
    \begin{array}{ll}
    x_a^{(j/M)},& \hbox{if}\; j/M\in\N,\\
    0,& \hbox{if}\; j/M\not\in\N,
    \end{array}
    \right.
    \end{equation}
    in the following sense.
    
    1. Whenever the vector function $X:=(x_0,x_1,\ldots, x_{N-1})^T$  is the solution of the system \eqref{eqn.reduction3}-\eqref{eqn.reduction4}, the function $x:=x_0$ solves the initial value problem \eqref{eqn.reduction1}-\eqref{eqn.reduction2} on $[a,b]$.
    
    2. Whenever the function $x\in C^{\lceil\alpha_k\rceil-1}$ is the solution the initial value problem \eqref{eqn.reduction1}-\eqref{eqn.reduction2}, the vector function 
    $$
    X:= (x_0, x_1, \ldots,x_{N-1})^T := (x, {}^{C}\!D_{a}^{\gamma}x,
    {}^{C}\!D_{a}^{2\gamma}x,\ldots, {}^{C}\!D_{a}^{(N-1)\gamma}x)^T
    $$ 
    solves the multidimensional initial value problem \eqref{eqn.reduction3}-\eqref{eqn.reduction4}.
   \end{theorem}
   \begin{proof}
   Note that the only difference between our theorem and Theorem 8.1 of Diethelm~\cite{Kai} is the assumption of smoothness of solution: we require the solution of the problem \eqref{eqn.reduction1}-\eqref{eqn.reduction2} to be continuously Caputo $\alpha_k$-differentiable whereas in Theorem 8.1 of \cite{Kai} one needs a stronger assumption of $C^{\lceil\alpha_k\rceil}[a,b]$. 
   
  It is easily seen that same arguments as in the proof of Theorem 8.1 of \cite{Kai} with the change of the usage of Lemma 3.13 there to the usage of our Theorem~\ref{thm.Cap3} furnish the proof of the theorem.
  \end{proof}
   
   \begin{corollary}\label{cor.1}
   In the situation of Theorem~\ref{thm.reduce}, assume additionally that the function $f$ is uniformly Lipschitz in the $x$ variable, i.e. there exists $L>0$ such that
   \begin{equation*}
    |f(t,y)-f(t,z)|\leq L\|y-z\|, \quad\hbox{for all}\quad t\in [a,b], \; y,z\in\R^k.
   \end{equation*}
   Then both the initial value problems  \eqref{eqn.reduction1}-\eqref{eqn.reduction2} and 
   \eqref{eqn.reduction3}-\eqref{eqn.reduction4} have unique solutions and the solutions agree in the sense of  Theorem~\ref{thm.reduce}.
   \end{corollary}
   
   As noted by Diethelm~\cite{Kai}, Theorem~\ref{thm.reduce1} has limitation in applications: it is applicable only to the case of rational indices of differentiation or commensurate indices if all the indices are not greater than 1. Furthermore, the reduced single-term system may be of very high dimension because the parameter $\gamma$ there can be very small. A way to tackle these drawbacks would be the use of  a multi-order Caputo fractional system instead of single-term Caputo fractional system what we present below.
   
   Let us consider Caputo fractional differential equation \eqref{eqn.reduction1} on $[a,b]$ subject to the initial conditions  \eqref{eqn.reduction2}. Put $\beta_1:=\alpha_1$, $\beta_j:= \alpha_j-\alpha_{j-1}$ for $j=2,\ldots,k$, $x_1:=x$, $x_j:= {}^{C}\!D_{a}^{\alpha_{j-1}}x$ for $j=2,\ldots,k$. Without loss of generality we assume further that there is no integer lying between two consecutive numbers $\alpha_j$, $j=1,\ldots,k$ because we may add an integer derivative of $x$ to the set of variables of $f$ without changing the function $f$ in \eqref{eqn.reduction1}. Note that by the assumption on $\alpha_j$, $j=1,\ldots, k$, we have $0<\beta_j\leq 1$ for all $j$. We have the following equivalence result
   
   \begin{theorem}[cf. Theorem 8.9 of Diethelm \cite{Kai}]\label{thm.reduce1}
  Subject to the above conditions, the multi-term equation \eqref{eqn.reduction1} with the initial conditions 
   \eqref{eqn.reduction2} is equivalent to the system
    \begin{equation}\label{eqn.reduction5}
    \begin{array}{rcl}
    {}^{C}\!D_{a}^{\beta_1}x_1(t) &=&x_2(t),\\
    {}^{C}\!D_{a}^{\beta_2}x_2(t) &=&x_3(t),\\
    &\vdots&\\
    {}^{C}\!D_{a}^{\beta_{k-1}}x_{k-1}(t) &=&x_{k}(t),\\
    {}^{C}\!D_{a}^{\beta_k}x_{k}(t) &=&
    f(t,x_1(t), x_2(t),
     \ldots, x_k(t))
    \end{array}
    \end{equation}
     with the initial conditions
     \begin{equation}\label{eqn.reduction6}
     x_j(a)=\left\{
    \begin{array}{ll}
    x_a^{(0)},& \hbox{if}\; j=1,\\
    x_a^{(l)},& \hbox{if}\; j=l\in\N,\\
    0,& \hbox{else},
    \end{array}
    \right.
    \end{equation}
    in the following sense.
   
   1. Whenever the function $x\in C^{\lceil\alpha_k\rceil-1}$ is a solution of the multi-term Caputo fractional differential equation \eqref{eqn.reduction1} with the initial conditions \eqref{eqn.reduction2}, the vector function $X:= (x_1, x_2, \ldots,x_{k})^T$ with 
     \begin{equation}\label{eqn.reduction7}
     x_j(t) :=\left\{
    \begin{array}{ll}
    x(t),& \hbox{if}\; j=1,\\
    {}^{C}\!D_{a}^{\alpha_{j-1}}x(t),& \hbox{if}\; j\geq 2,
      \end{array}
    \right.
    \end{equation}
    is a solution of the multi-order Caputo fractional differential equation \eqref{eqn.reduction5} with initial conditions  \eqref{eqn.reduction6}.
    
    2. Whenever the vector function $X:=(x_1,x_2,\ldots, x_{k})^T$  is a solution of the multi-order Caputo fractional differential system \eqref{eqn.reduction5} with initial conditions \eqref{eqn.reduction6}, the function $x:=x_1$ solves the multi-term Caputo fractional differential equation \eqref{eqn.reduction1} with the initial conditions \eqref{eqn.reduction2} on $[a,b]$.
   \end{theorem}
   
   \begin{proof}
   Similar to the proof of Theorem \ref{thm.reduce} above.
   \end{proof}

   \begin{remark}\label{rem.reduce}
   (i) Theorems~\ref{thm.reduce} and \ref{thm.reduce1} are still valid if we considers the equations on the infinite time interval $[a,\infty)$.
   
   (ii) In the situation of Theorem~\ref{thm.reduce}, if $\alpha_k\leq 1$ then one needs not assume the indices $\alpha_j$, $j=1,\ldots, k$ to be rational numbers but only the assumption that they are commensurate, i.e. $\alpha_i/\alpha_j\in\Q$ for all $j,j$, suffices.
   
   (iii) In Theorems~\ref{thm.reduce} and \ref{thm.reduce1}  we do not assume neither existence nor the uniqueness of solutions of the systems in the considerations.
   
   (iv) The smoothness conditions on solutions of the systems considered in Theorems~\ref{thm.reduce} and \ref{thm.reduce1} are natural enough such that it is automatically provided by all the theorems on the existence and uniqueness of the solutions of those systems under a natural requirement of Lipschitz continuity of the coefficients of the systems.
   \end{remark}
   
   \subsection{Existence and uniqueness of solution to multi-term Caputo fractional differential systems}
   
  Based on our results above and the available partial results in the literature we are able to give a concise  proof of a theorem on existence and uniqueness  of solutions to a general multi-term and general multi-order Caputo fractional differential system. Notice that a theorem on existence and uniqueness of solution to multi-term is proved by Diethelm~\cite[Chapter 8]{Kai} using a fairly complicated method of approximation.
   
   \begin{theorem}\label{thm.reduce2}
     Consider the equation
    \begin{equation}\label{eqn.reduction8}
    {}^{C}\!D_{a}^{\alpha_k}x(t) = f(t,x(t), {}^{C}\!D_{a}^{\alpha_1}x(t),
    {}^{C}\!D_{a}^{\alpha_2}x(t), \ldots, {}^{C}\!D_{a}^{\alpha_{k-1}}x(t)),
    \end{equation}
  $t\in [a,b]$,  subject to the initial conditions
    \begin{equation}\label{eqn.reduction9}
    (D^jx)(a) = x_a^{(j)},
    \quad j=0,1,\ldots \lceil \alpha_k\rceil -1,
    \end{equation}
    where $\alpha_k>\alpha_{k-1}>\cdots >0$, $\alpha_j-\alpha_{j-1}\leq 1$ for all $j=2,3,\ldots,k$ and $0<\alpha_1\leq 1$, $f: [a,b]\times \R^k \to\R$ is a continuous function. Assume that $f$ satisfies an uniform Lipschitz condition, i.e. there exists a positive constant $L>0$ such that
    \begin{equation*}
    |f(t,y)-f(t,z)|\leq L\|y-z\|, \quad\hbox{for all}\quad t\in [a,b], \; y,z\in\R^k.
    \end{equation*}
    Then for any $x_0\in\R^k$ there exists unique continuous solution $x:[a,b] \to \R^k$ to the equation \eqref{eqn.reduction8} that satisfies the initial conditions  \eqref{eqn.reduction9}.
       \end{theorem}
       
       \begin{proof}
       Notice that the system \eqref{eqn.reduction8}-\eqref{eqn.reduction9} is exactly the system
       \eqref{eqn.reduction1}-\eqref{eqn.reduction2} treated in Theorem~\ref{thm.reduce1}.
       Use Theorem~\ref{thm.reduce1} to reduce the multi-term Caputo fractional differential system 
       \eqref{eqn.reduction8}-\eqref{eqn.reduction9} to a multi-order system  \eqref{eqn.reduction5}-\eqref{eqn.reduction6}. By Theorem~\ref{thm.reduce1} all the orders of Caputo differentiation in the derived multi-order Caputo fractional differential system \eqref{eqn.reduction5}-\eqref{eqn.reduction6} are in $(0,1]$. Therefore, due to the assumption of Lipschitz continuity of $f$,  Theorem 2.3 of Diethelm, Siengmund and Tuan~\cite{DieSieTuan2017} is applicable to the multi-order system \eqref{eqn.reduction5}-\eqref{eqn.reduction6}, implying that for any initial value \eqref{eqn.reduction6} the multi-order system \eqref{eqn.reduction5} has unique solution. By Theorem~\ref{thm.reduce1}, this implies that for all initial value  \eqref{eqn.reduction9} the original multi-order system \eqref{eqn.reduction8} has unique solution. 
       \end{proof}

   \section{Illustrative examples}\label{sec.examples}
In this section we give two examples. The first is an example of a diagonal system to show non-applicability of Theorems of Diethelm~\cite[Theorem 8.1]{Kai} due to the strong smoothness requirements, and in the meantime Theorem~\ref{thm.reduce} is applicable to that system. The second example is application of the reduction technique to investigate Lyapunov stability of a linear multi-order Caputo fractional system; this example is taken from \cite{DieSieTuan2017}.
 
 \begin{example}[$\alpha$-differentiability but no $C^{\lceil\alpha\rceil}$-smoothness]
 Consider the multi-term Caputo FDE,
 \begin{equation}\label{eqn.ex1}
 {}^C\!D_{0}^{\alpha} x - \lambda{}^C\!D_{0}^{\beta} x=0,
 \end{equation}
 $t\in [0,2]$, $0<\beta<\alpha$. Choose $\lambda=1$, $\beta=1, \alpha=1.5$ so $\alpha-\beta=0.5 =:\gamma$, $\beta=2\gamma$, $\alpha=3\gamma$.
 According to Kilbas et al.~\cite[Theorem 5.13, p. 314]{KilbasSriTru06} this equation has solution
 \begin{eqnarray*}
 x_0(t) &=& E_{\alpha-\beta}(\lambda t^{\alpha-\beta}) -\lambda t^{\alpha-\beta}
 E_{\alpha-\beta,\alpha-\beta+1}(\lambda t^{\alpha-\beta})=1,\\
 x_1(t)&=& tE_{\alpha-\beta,2}(\lambda t^{\alpha-\beta})= t E_{\gamma,2}(t^{\gamma}).
 \end{eqnarray*}
 We have $x_0(0)=1, Dx_0(0)=0$ and $x_1(0)=0, Dx_1(0)=1$. These two solutions form a fundamental system of solutions to the linear equation \eqref{eqn.ex1}.
 
 A simple computation gives  
 $Dx_1= E_\gamma(t^\gamma) + E_{\gamma,\gamma}(t^\gamma) t^\gamma$ on $[0,2]$, and for $0<t\leq 2$
 $$
 D^2 x_1 =
 E_{\gamma,\gamma}(t^\gamma) t^\gamma + 
( DE_{\gamma,\gamma}(t^\gamma)) t^\gamma + E_{\gamma,\gamma}(t^\gamma) \gamma t^{\gamma -1},
 $$
 hence $x_1\in C^1[0,2]\cap C^2(0,2]$ but $x_1\not\in C^2[0,2]$. 
 
 Now we make a change of variables to apply Theorem~\ref{thm.reduce}  to reduce the Caputo fractional equation \eqref{eqn.ex1} to the case of linear system of single-term Caputo fractional differential equations. Put
   $$
   y_1(\cdot) := x(\cdot), \quad y_2(\cdot) := {}^C\!D_{0}^{\gamma} y_1, \quad y_3(\cdot) := {}^C\!D_{0}^{\gamma}y_2(\cdot).
$$   
   Then  the FDE \eqref{eqn.ex1} becomes
   \begin{equation}\label{eqn.ex2}
\left\{ \begin{array}{ccl}
{}^C\!D_{0}^{\gamma} y_1(t)&=&y_2(t),\\[3pt]
{}^C\!D_{0}^{\gamma} y_2(t)&=&y_3(t),\\[3pt]
{}^C\!D_{0}^{\gamma} y_3(t)&=&y_3(t),
\end{array}
\right. 
\end{equation}
   or, in the vector form,
   $$
   {}^C\!D_{0}^{\gamma} y = A y; \quad 
   y:=  \left( \begin{array}{c} y_1\\y_2\\y_3\end{array}\right)\in \R^3,\quad
    A =
 \left(
      \begin{array}{ccc}
0 & 1 & 0\\
      0  & 0&     1             \\
                0& 0 & 1  
      \end{array}
    \right) \in \R^{3\times 3}.
   $$
    The characteristic polynomial of the matrix $A$ of the linear time independent single-term Caputo fractional differential equation \eqref{eqn.ex2} is
    \begin{equation}\label{eqn.ex3}
   \Delta_A(s)=  
   \det ({\rm diag}(\lambda I-A) =
   \lambda^3 - \lambda^2.
   \end{equation}
  The matrix $A$  has eigenvalue $0$ with multiplicity 2 and 1 with multiplicity 1. Hence,  by 
  Diethelm~\cite[Theorem 7.14, p. 151]{Kai} the three-dimensional linear time independent Caputo fractional differential equation \eqref{eqn.ex2} has general solution with the first coordinate given by
     $$
     y_1(t)= a_1+a_2t^\gamma+a_3 E_\gamma(t^\gamma), \quad a_1,a_2,a_3\in\R.
     $$
     Since $0<\gamma<1$ we must have ${}^C\!D_{0}^{\gamma} y_1(0)=0$ hence $a_3=-a_2\Gamma(1+\gamma)$. Therefore the general solution of \eqref{eqn.ex2} is
   $$
  \left\{ \begin{array}{rcl} y_1(t) &=& a_1 +a_2 t^\gamma -
  a_2\Gamma(1+\gamma) E_\gamma(t^\gamma),   \\[3pt]
  y_2(t) &=& a_2\Gamma(1+\gamma) -  a_2\Gamma(1+\gamma)E_\gamma(t^\gamma), \\[3pt] 
  y_3(t) &=& -a_2\Gamma(1+\gamma) E_\gamma(t^\gamma),
   \end{array}\right. 
  $$
  for arbitrary $a_1,a_2\in\R$. By choosing $(a_1=1, a_2=0)$ and $(a_1=0, a_2=-1/\Gamma(1+\gamma))$ we get two vectors
  
  $$
  {\hat y}_1(t)= \left( \begin{array}{c} 1\\0\\0\end{array}\right),\quad
  {\hat y}_2(t)= \left( \begin{array}{c}-1 -t^\gamma+ \Gamma(1+\gamma)E_\gamma(t^\gamma)\\
 -\Gamma(1+\gamma)+ \Gamma(1+\gamma)E_\gamma(t^\gamma)\\
 \Gamma(1+\gamma)E_\gamma(t^\gamma)\end{array}\right)
  $$
  forming a fundamental system of solution of \eqref{eqn.ex2}. Notice that all the solution of \eqref{eqn.ex2} are continuously Caputo fractional $\alpha$-differentiable, but the first coordinate of the solution 
  ${\hat y}_2(t)$ is does not belong to $C^2[0,2]$, hence Theorem 8.1 of Diethelm~\cite{Kai} is not applicable to the solution ${\hat y}_2$ of \eqref{eqn.ex2}, whereas  Theorem~\ref{thm.reduce} is applicable to the solution ${\hat y}_2$ of \eqref{eqn.ex2} allowing us to conclude that the first coordinate of the vector ${\hat y}_2$ is a solution of the equation \eqref{eqn.ex1}.
  
  By the way one may verify that the solutions given by the application of Theorem~\ref{thm.reduce} and the  ones given by the direct computation using \cite[Theorem 5.13, p. 314]{KilbasSriTru06} above agree, i.e.
  $$
  -1 -t^\gamma+ \Gamma(1+\gamma)E_\gamma(t^\gamma) = t E_{\gamma,2}(t^{\gamma}).
  $$ 
  This equality can be verified by using series presentation of $E_\gamma(t^\gamma)$ and $E_{\gamma,2}(t^{\gamma})$.
 \end{example}
 
 \begin{example}[Stability study by reduction to a single-term system]
 Recall Example 3.6 of Diethelm, Siegmund and Tuan~\cite{DieSieTuan2017}. Consider the system
 \begin{equation}\label{eqn.ex4}
   y:=  \left( \begin{array}{c} {}^C\!D_{0}^{1/2}x_1(t)\\{}^C\!D_{0}^{1/4}x_2(t)\end{array}\right)= Ax(t),\quad\hbox{where}\quad
    A =
 \left(
      \begin{array}{cc}
0.00001 & 1 \\
     - 0.0022  & 0.1  
      \end{array}
    \right) .
 \end{equation}
 This two-dimensional system can be written  as a three-dimensional system of order $1/4$ in the form
  \begin{equation}\label{eqn.ex5}
  {}^C\!D_{0}^{1/4}x^*(t)= A^*x^*(t),\quad\hbox{with}\quad 
   A^* =
 \left(
      \begin{array}{ccc}
0 & 1 & 0\\
      0.00001  & 0&     1             \\
                -0.0022& 0 & 0.1  
      \end{array}
      \right).
 \end{equation}
 By Theorem \ref{thm.reduce1} system \eqref{eqn.ex4} is equivalent to the system \eqref{eqn.ex5}. Since 
 $A^*$ has eigenvalues $-0.1.03917$ and $0.101958\pm 0.10385i$,
  the system \eqref{eqn.ex5} is stable we conclude that the system \eqref{eqn.ex4} is stable (see \cite{DieSieTuan2017}). Notice that the stability of \eqref{eqn.ex4} is not easily deducted without its reduction to the single-term system \eqref{eqn.ex5}. Furthermore we also note that the solution of \eqref{eqn.ex5} is not of the class $C^1$, hence Theorem 8.1 of Diethelm \cite{Kai} is not application to this situation.
    \end{example}
   
   \section*{Acknowledgement}
This research is funded by the Vietnam National Foundation for
Science and Technology Development (NAFOSTED).
%


\begin{thebibliography}{15}
%
\bibitem{Badri2019}
P. Badri and M. Sojoodi.
\newblock{Stability and Stabilization of Fractional‐Order Systems with Different Derivative Orders.}
\newblock{\em Asian Journal of Control,} {\bf 21}(2019), 1--10.
%
\bibitem{Caputo67}
M. Caputo.
Linear models of dissipation whose Q is almost frequency independent – II.
{\em Geophys. J. Roy. Astron. Soc.}, Vol. 13 (1967), 529--539; reprinted in Fract. Calc. Appl. Anal. 11, 4–14 (2008)
%
\bibitem{Deng2007}
W. Deng, C. Li and Q. Guo.
\newblock{Analysis of fractional differential equations with multi-orders.} {\em
Fractals}, {\bf 15}(2007), 173--182.
%
\bibitem{Kai}
K.~Diethelm.
\newblock{\em The Analysis of Fractional Differential Equations. An Application-oriented Exposition Using Differential Operators of Caputo Type.}
\newblock{\em  Lecture Notes in Mathematics} {\bf 2004}.
\newblock{Springer-Verlag, Berlin, 2010.}
%
\bibitem{DieSieTuan2017}
K. Diethelm, S. Siegmund and H. T. Tuan.
\newblock{Asymptotic behavior of solutions of linear multi-order fractional differential systems.}
\newblock{\em Fractional calculus and applied analysis,} {\bf 20}(2017), 1165--1195.
%
\bibitem{Gorenflo2014}
R.~Gorenflo, A. A. Kilbas, F. Mainardi, S. V. Rogosin.
\newblock{\em Mittag-Leffler Functions, Related Topics and Applications.}
\newblock{Springer-Verlag, Berlin, 2014.}
%
\bibitem{KilbasSriTru06}
A. A. Kilbas, H. M. Srivastava and J. J. Trujillo.
\newblock{\em Theory and Applications of  Fractional Differential Equations. North-Holland mathematics studies, Vol. 204.}
\newblock{Elsevier, Amsterdam, 2006.}
%
\bibitem{Li2013}
C. Li, F. Zhang, J. Kurths and F. Zeng.
\newblock{Equivalent system for a multiple-rational-order fractional differential system.}
\newblock{\em Phil. Trans. R. Soc. A} 371:20120156, (2013), http://dx.doi.org/10.1098/rsta.2012.0156.
%
\bibitem{MillerRoss93}
K. S. Miller, B. Ross.
\newblock{\em An Introduction to the Fractional Calculus and Fractional Differential Equations.}
\newblock{ Wiley, New York, 1993.}
\bibitem{Podlubny}
I.~Podlubny.
\newblock{\em Fractional Differential equations. An Introduction to Fractional Derivatives, Fractional Differential Equations, to Methods of their Solution and some of Their Applications.}
\newblock{ Mathematics in Science and Engineering, \textbf{198}.}
\newblock{ Academic Press, Inc., CA, 1999.}
%
\bibitem{Rebenda2019}
J. Rebenda.
\newblock{Application of Differential Transform to Multi-Term Fractional Differential Equations with Non-Commensurate Orders.}
\newblock{\em Symmetry,} {\bf 11}(2019), 1390.
%
\bibitem{Vainikko2016}
G. Vainikko.
\newblock{Which functions are fractionally differentiable?}
\newblock{\em Zeitschrift f\"ur Analysis und ihre Anwendungen.}
Volume 35 (2016), 465--487.
\end{thebibliography}
\end{document}